\documentclass[11pt]{article}
\usepackage{amsmath,a4,bm,color}
\usepackage{amsthm}
\usepackage{amsmath}
\usepackage{amsthm}
\usepackage{amscd}
\usepackage{amssymb}
\usepackage{empheq}
\newtheorem{thm}{Theorem}
\numberwithin{equation}{section}

\newtheorem{lem}{Lemma}
\newtheorem{cor}{Corollary}
\newtheorem*{cor*}{Corollary}

\theoremstyle{definition}

\theoremstyle{remark}
\newtheorem{rem}{Remark}
\theoremstyle{remark}

\newcommand{\cQ}{{\cal Q}}
\newcommand{\cU}{{\cal U}}

\newcommand{\cS}{{\cal S}}

\newcommand{\cM}{\mathcal{M}}

\newcommand{\vv}[1]{\mathbf{#1}}
\newcommand{\Z}{\mathbb{Z}}
\newcommand{\Q}{\mathbb{Q}}
\newcommand{\R}{\mathbb{R}}

\newcommand{\N}{\mathbb{N}}

\newcommand{\hs}{\mathcal{ H}^{s}}
\newcommand{\cH}{{\cal H}}

\begin{document}

\title{\bf Diophantine approximation on manifolds and lower bounds for Hausdorff dimension}

\author{Victor Beresnevich\footnote{Supported by EPSRC Programme Grant: EP/J018260/1}\and Lawrence Lee\footnote{Supported by LMS Undergraduate Research Bursary URB 16-39}\and Robert C. Vaughan \and \\ Sanju Velani\footnote{Supported by EPSRC Programme Grant: EP/J018260/1} \\ }

\date{{\small\it Dedicated to Klaus Roth}}

\maketitle

\begin{abstract}
Given $n\in\N$ and $\tau>\frac1n$, let $\cS_n(\tau)$ denote the classical set of $\tau$-approximable points in $\R^n$, which consists of $\vv x\in \R^n$ that lie within distance $q^{-\tau-1}$ from the lattice $\frac1q\Z^n$ for infinitely many $q\in\N$.  In pioneering work, Kleinbock $\&$ Margulis showed that for any non-degenerate submanifold $\cM$ of $\R^n$ and any $\tau>\frac1n$ almost all points on $\cM$ are not $\tau$-approximable. Numerous subsequent papers have been geared towards strengthening this result  through  investigating   the Hausdorff measure and dimension of the associated null set  $\cM\cap\cS_n(\tau)$. In this paper we suggest a new approach based on the Mass Transference Principle of \cite{Mass}, which enables us to find a sharp lower bound for  $\dim \cM\cap\cS_n(\tau)$ for any $C^2$ submanifold $\cM$ of $\R^n$
and any $\tau$ satisfying $\frac1n\le\tau<\frac1m$. Here $m$ is the codimension of  $\cM$. We also show that the condition on $\tau$ is best possible and extend the result to general approximating functions.
\end{abstract}

\section{Introduction}

Throughout $\psi:\N\to\R^+$ will denote a monotonic function that will be referred to as an \emph{approximating function}, $n\in\N$ and $\cS_n(\psi)$ will be the set of \emph{$\psi$-approximable} points $\vv y\in\R^n$, that is points $\vv y=(y_1,\dots,y_n)\in\R^n$ such that
\begin{equation}\label{vb1}
\max_{1\le i\le n}|qy_i-p|<\psi(q)
\end{equation}
holds for infinitely many integer vectors $(\vv p,q)=(p_1,\dots,p_n,q)\in\Z^n\times\N$.
Thus the point $\vv y$ lies within distance $\psi(q)/q$ (in the supremum norm) from the lattice $\frac1q\Z^n$ for infinitely many $q\in\N$.

If $\psi(q)=q^{-\tau}$ for some $\tau>0$, then we write $\cS_n(\tau)$ for $\cS_n(\psi)$ and say that $\vv y\in\cS_n(\tau)$ is~\emph{$\tau$-approximable}. The classical theorem of  Dirichlet states that for any
$\vv y\in\R^n$, there exist infinitely many $(\vv p,q)=(p_1,\dots,p_n,q)\in\Z^n\times\N$ such that
\begin{equation*}\label{svb1}
\max_{1\le i\le n}|qy_i-p_i|<  q^{-\frac1n}\quad\quad\ (1\le i\le n)  \, .
\end{equation*}
Hence,
\begin{equation} \label{dualexp+}
\cS_n(\tau) = \R^n \quad\text{for any }\tau\le \tfrac1n\,.
\end{equation}
In turn, a rather simple consequence of the Borel-Cantelli lemma from probability theory is that $\cS_n(\psi)$ is null (that is of $n$-dimensional Lebesgue measure zero) whenever $\sum_{q=1}^\infty\psi^n(q)<\infty$. Thus, for any  $\tau > 1/n$ almost every point $\vv y\in\R^n$ is not $\tau$-approximable.  However, Khintchine's theorem \cite{Khintchine-1924} tells us that the set $\cS_n(\psi)$ is full (that is its complement is of Lebesgue measure zero) whenever $\sum_{q=1}^\infty\psi^n(q)=\infty$. In order to quantify the size of $\cS_n(\psi)$ when it is null, Jarn\'ik \cite{Jarpa} and Besicovitch \cite{Besic34} pioneered the use of Hausdorff measures and dimension. Throughout, $\dim X$ will denote the Hausdorff dimension of a  subset $X$ of  $\R^n$ and $\cH^s(X)$ the $s$-dimensional Hausdorff measure (see \S\ref{mi} for the definition and further details).   The modern version of the classical Jarn\'ik-Besicovitch theorem (see \cite{BBDVRoth,BDV06}) states that for any approximating function $\psi$
\begin{equation}\label{JBdani}
\dim \cS_n(\psi) = \min\left\{n,\frac{n+1}{\tau_{\psi}^*+1}\right\}   \qquad\text{where }  \quad \tau_{\psi}^*:=\liminf_{q\to\infty}\frac{-\log\psi(q)}{\log q}\,.
\end{equation}
In other words, the `modern theorem'  relates the Hausdorff dimension of  $\cS_n(\psi)$ to  the lower order  at infinity of $1/\psi$ and up to a certain degree allows us to  discriminate between $\psi$-approximable sets of Lebesgue measure zero. The classical Jarn\'ik-Besicovitch theorem corresponds to  \eqref{JBdani} with $\cS_n(\psi)$  replaced by  $\cS_n(\tau)$  and so by definition  $ \tau_{\psi}^* = \tau$.
A more delicate measurement of the `size' of  $\cS_n(\psi)$ is obtained by expressing the size in terms of Hausdorff measures $\cH^s$.   With respect to such measures, the modern version of Jarn\'ik theorem  (see \cite{BBDVRoth} or \cite{BDV06})  states that  for any $s \in (0,n)$ and any  approximating function  $\psi$
\begin{equation}\label{jarnikdani}
%\cH^{s}\big(\cS_n(\psi)\cap\I\big) \, = \,
\cH^{s}\big(\cS_n(\psi)\big) \, = \,
\left\{\begin{array}{cl}
0  & {\rm if} \;\;\;
\textstyle{\sum_{q=1}^\infty} \; q^{n-s}\psi^{s}(q) \; <\infty \; ,\\[3ex]
\infty & {\rm if} \;\;\;
\textstyle{ \sum_{q=1}^\infty } \;  q^{n-s}\psi^{s}(q) \;
 =\infty \; .
\end{array}\right.
\end{equation}
%Note that for $0<s<n$ we have that $\cH^s(\I)=\infty$. However,  since $\cH^n(\I)=1$,  the statement as written also holds for $s=1$ due to the aforementioned theorem of Khintchine.   Note that it is trivially true for $s>n$. The upshot is that statement \eqref{jarnikdani} is true for any $s >0 $  and is referred to as the {\em Khintchine-Jarn\'ik theorem}.
This statement is a natural generalisation of Khintchine's theorem (a Lebesgue measure statement) to Hausdorff measures and it is easily verified that it implies \eqref{JBdani}.
It is worth pointing out that there is an even more general version of \eqref{jarnikdani} that makes use of more general Hausdorff measures, see \cite{BBDVRoth,  BDV06, Mass,det}. Within this paper we restrict ourselves to the case of $s$-dimensional Hausdorff measures.  For  background and further details  regarding the classical theory of metric Diophantine approximation with an emphasis on the statements described  above  see   \cite{durham}

When the coordinates of the approximated point $\vv y \in \R^n$ are confined by functional relations, we fall into the theory  of Diophantine approximation on manifolds.    In short, given a manifold $\cM$ of $\R^n$ the aim is to establish analogues of the  fundamental theorems of Khintchine, Jarn\'ik-Besicovitch and Jarn\'\i k, and thereby provide a complete measure theoretic description of the sets $\cS_n(\psi) \cap\cM $. The fact that the points $\vv y \in \R^n$ of
interest are of dependent variables, which reflects the fact that
$\vv y \in {\cal M}$, introduces major difficulties in attempting to
describe the measure theoretic structure of $  \cS_n(\psi) \cap\cM
$.

A differentiable manifold $\cM$ of $\R^n$ is called \emph{extremal} if for any $\tau>\tfrac1n$ almost every $\vv y\in\cM$ (in the sense of the induced Lebesgue measure on $\cM$) is not $\tau$-approximable \cite{BD00}. Thus, $\cM=\R^n$ is extremal. In 1998 Kleinbock $\&$ Margulis \cite{KM98} proved Sprind\v zuk's conjecture; namely that any non-degenerate submanifold of $\R^n$ is extremal. Essentially, these are smooth
submanifolds of $\R^n$ which are sufficiently curved so as to deviate
from any hyperplane. Any real, connected analytic manifold not
contained in any hyperplane of $\R^n$ is non--degenerate.
The extremality result  of \cite{KM98} has subsequently been  extended to include various degenerate  manifolds and more generally subsets of $\R^n$ that support so called friendly measures, see  \cite{Kle03, KLW, PV} and references within. Indeed, over the last decade or so, the theory of Diophantine approximation on manifolds has developed at  some considerable pace with the catalyst being the pioneering work of Kleinbock \& Margulis .  For details of this and an overview of the current state of the more subtle  Khintchine and Jarn\'ik type results for  $\cS_n(\psi) \cap \cM $ see \cite[\S6]{durham} and references within.

The  main goal of this paper is to establish  Jarn\'ik-Besicovitch type  results for manifolds.  In other words, we are interested in the Hausdorff dimension of $\cS_n(\tau)\cap\cM$ for $\tau > \frac1n$. Note that in view of \eqref{dualexp+}, for $\tau  \le  \frac1n$ we always have that $\cS_n(\tau)\cap\cM = \cM$  and so there is nothing more to say.   Throughout, let $m$ denote  the codimension of the differentiable  submanifold $\cM$ of $\R^n$. Then, heuristic considerations quickly lead to the following formula:
\begin{equation}\label{JB}
\dim\cS_n(\tau)\cap\cM \, = \, \frac{n+1}{\tau+1}-m\,  .
\end{equation}
In the case $\cM=\R^n$ (so that $m=0$)  the above formula coincides with  \eqref{JBdani} and holds for all $\tau\ge\frac1n$. Indeed, it corresponds to the classical  Jarn\'ik-Besicovitch theorem.  For non-degenerate curves in $\R^2$, as a result of various works \cite{BDV, BZ, Huang, VV06}, we know  that \eqref{JB} holds for all $\tau\in[\tfrac12,1)$.  For  completeness it is worth pointing out that non-degenerate planar curves are characterised by being $C^2$ and having non-zero curvature.

Beyond planar curves, for analytic non-degenerate manifolds $\cM$  in $\R^n$ the lower bound associated with \eqref{JB} is established in  \cite{B12} and holds for $\tau \in [\frac1n , \frac1m)$. The analytic condition is removed in the case of  curves and manifolds that can be
  `fibred' into non-degenerate  curves \cite{BVVZdiv}.  The latter includes
  $C^{\infty}$ non-degenerate submanifolds of $\R^n$ which are not
  necessarily analytic.   The complementary upper bound associated with \eqref{JB} has  been established \cite{BVVZcon, Simmons} under additional geometric restrictions on $\cM$ and  a more limited range of $\tau$. The difficulty in obtaining upper bounds lies in the notoriously difficult problem of bounding the number of rational points lying near the manifold of interest. For a discussion of this including the heuristics see \cite[\S6.1.3]{durham}.

We emphasize that all the lower bound dimension results mentioned in the above discussion are actually derived from  corresponding   stronger  divergent Jarn\'ik-type results for  $ \cH^s(\cS_n(\psi)\cap\cM) $. The main substance of this paper is the introduction of a new approach, based on the Mass Transference Principle of \cite{Mass},  that is primarily geared towards establishing dimension statements. As a consequence we are able to prove the following general result,  not just for non-degenerate manifolds but for any $C^2$ submanifold of $\R^n$.  Furthermore, the  proof is significantly easier than those based on establishing a divergent Jarn\'ik-type result.

\bigskip

\begin{thm}\label{t0}
Let $\mathcal{M}$ be any $C^2$ submanifold of $\R^n$ of codimension $m$ and let
\begin{equation}\label{tau}
\tfrac1n\le\tau<\tfrac1m.
\end{equation}
Then
\begin{equation}\label{v5}
\dim \mathcal{S}_{n}(\tau)\cap\mathcal{M} ~\geq~ s:=\frac{n+1}{\tau+1}-m\,.
\end{equation}
Furthermore,
\begin{equation}\label{vb770}
  \cH^s(\mathcal{S}_{n}(\tau)\cap\mathcal{M})=\cH^s(\cM)\,.
\end{equation}
\end{thm}

\medskip

\noindent Before we state the  more general result for the Hausdorff dimension of $\cS_n(\psi)  \cap \cM$ with $\psi$ not necessarily of the form $\psi(q):=q^{-\tau}$, we introduce the Mass Transference Principle and  give a `simple'  proof of Theorem~\ref{t0} in the special case that $\cM$ is an affine coordinate plane. This will bring to the forefront the main ingredients of the  new approach developed in this paper.  It is worth mentioning that for such planes (which are clearly degenerate) Khintchine-type results have recently been established \cite{Felipe1,Felipe2}  -- see also \cite[\S4.4]{durham}.

\section{Introducing the main ingredients \label{mi}}

%Before we state the  more general result for the dimension of $\cS_n(\psi)  \cap \cM$ with $\psi$ not necessarily of the form $\psi(q):=q^{-\tau}$, we use this section to introduce the Mass Transference Principle and to give a `simple'  proof of Theorem~\ref{t0} in the special case that $\cM$ is an affine coordinate plane.  It is worth mentioning that for such planes (which are clearly degenerate) Khintchine-type results have recently been established \cite{Felipe1,Felipe2}  -- see also \cite[\S4.4]{durham}.

First some notation. Throughout, $\R^k$ will be regarded as a metric space with distance induced by any fixed norm (not necessarily Euclidean) and $B(\vv x,r)$ will denote a ball centred at $\vv x \in \R^k$ of radius $r>0$. Given a ball $B=B(\vv x,r)$ and a positive real number $\lambda$, we denote by  $\lambda
 B$  the ball $B$ scaled by the factor $\lambda$; i.e.  $\lambda B(\vv x,r):= B(\vv x, \lambda r)$.
Also, given $s>0$ and a ball $B=B(\vv x,r)$ in $\R^k$, we define another ball
\begin{equation}\label{e:006}
\textstyle B^s:=B(\vv x,r^{s/k}) \ .
\end{equation}
Note that trivially $B^k=B$.

 Suppose $F$ is a
subset of $\R^k$. Given a ball $B$ in $\R^k$, let ${\rm diam}(B)$ denote the
diameter of $B$. For $\rho > 0$, a countable collection
$ \left\{B_{i} \right\} $ of balls in $\R^k$ with ${\rm diam}(B_i) \leq \rho $
for each $i$ such that $F \subset \bigcup_{i} B_{i} $ is called a {\em
  $ \rho $-cover for $F$}. Given a real number $s \ge 0$, the \emph{Hausdorff $s$-measure} of $F$ is given by
$$ {\cal H}^{s} (F) :=
\lim_{ \rho \rightarrow 0} \inf \ \sum_{i} {\rm diam}(B_i)^s\,,
$$
where the infimum is taken over all $\rho$-covers  of $F$.
When $s$ is an integer, $\hs$ is a constant multiple of
$s$--dimensional Lebesgue measure.
The
\emph{Hausdorff dimension} of $F$ is defined by $$ \dim \,
F \, := \, \inf \left\{ s : {\cal H}^{s} (F) =0 \right\}  . $$
Further details regarding
      Hausdorff measure and dimension can be found
      in~\cite{Falc}.

Observe that the set $\cS_n(\psi)$ of $\psi$-well approximable points in $\R^n$ is the $\limsup$ of the sequence of hypercubes in $\R^n$ defined by \eqref{vb1}. Recall that, by definition, given a sequence of sets $(S_i)_{i\in I}$ indexed by a countable set $I$, then
$$
\limsup S_i:=\{\vv x\in S_i:\text{for infinitely many }i\in I\}\,.
$$
The following transference theorem concerning $\limsup $ sets is the key to establishing the results of this paper.

\medskip

\begin{thm}[Mass Transference Principle]\label{thm3}
Let $\cU$ be an open subset of $\R^k$. Let $\{B_i\}_{i\in\N}$ be a sequence of balls in $\R^k$ centred in $\cU$ with ${\rm diam}(B_i)\to 0$ as $i\to\infty$. Let $s>0$ and suppose that for  any ball
$B$ in $\cU$
\begin{equation}\label{e:011}
\cH^k\big(\/B\cap\limsup_{i\to\infty}B^s_i{}\,\big)=\cH^k(B) \ .
\end{equation}
Then, for any ball $B$ in $\cU$
\begin{equation*}%\label{e:012}
\cH^s\big(\/B\cap\limsup_{i\to\infty}B^k_i\,\big)=\cH^s(B) \ .
\end{equation*}
\end{thm}

\bigskip

\noindent The above version of the Mass Transference Principle is easily  deduced from the original statement appearing as Theorem~2 in \cite{Mass}.  %As shown in \cite\S3.2 of

\medskip

Armed with the Mass Transference Principle it is easy to deduce the Jarn\'ik-Besicovitch theorem directly from Dirichlet's theorem.   For reasons that will soon become apparent  we provide the details.  To begin with, observe that $\vv y=(y_1,\dots,y_n)\in \mathcal{S}_{n}(\tau)$  if and only if
\begin{equation}\label{bpqsv}
\vv y\in  B_{\textbf{p},q}:= \left\lbrace\boldsymbol{\vv x}=(x_{1},\dots,x_{n})\in\mathbb{R}^n: \max_{1\le i\le n} \left|x_{i}-\frac{p_{i}}{q}\right|<q^{-\tau-1}\right\rbrace  \,
\end{equation}
for infinitely many $ q \in \N$ and $\vv p=(p_1,\dots,p_n)\in\Z^n$; that is
$$
 \mathcal{S}_{n}(\tau)  = \limsup_{q\to\infty} B_{\vv p,q}  \, .
$$
Next, in view of Dirichlet's theorem we have that
$$
 \limsup_{q\to\infty} B^s_{\vv p,q} =  \R^n  \quad {\rm where }  \quad    s:=\frac{n+1}{\tau+1} \,
$$
and  the ball  $B^s_{\vv p,q}$   associated with $B_{\vv p,q}$ is defined via \eqref{e:006} with $k=n$.    Thus, for any ball  $B$ in $\R^n$
 it  trivially follows that $$
\cH^n(B\cap \limsup_{q\to\infty} B^s_{\vv p,q})=\cH^n(B) \, .
$$
In turn, for any $\tau > 1/n$ (so that $ s < n $)  the Mass Transference Principle implies that
$$
\cH^s(\mathcal{S}_{n}(\tau))=\cH^s(\R^n) = \infty  \, .
$$
Hence, by the definition of Hausdorff dimension, $ \dim \mathcal{S}_{n}(\tau)  \ge s $. This is the hard part in establishing the Jarn\'ik-Besicovitch theorem.  The  complementary upper bound is pretty straightforward -- see for example \cite[\S2]{durham}. Observe that we actually proved a lot more than simply the Jarn\'ik-Besicovitch theorem.  We have shown that the Hausdorff $s$-measure of $\mathcal{S}_{n}(\tau)$ at the critical exponent is infinite.

We shall now see that to establish Theorem~\ref{t0} in the case of affine coordinate planes is essentially as easy as the proof of the Jarn\'ik-Besicovitch theorem just given.  To the best of our knowledge, even this restricted version of Theorem~\ref{t0} is new and can be thought of as the  `fibred' version of the classical Jarn\'ik-Besicovitch theorem.

%
%We are now in a position to give a short proof of Theorem~\ref{t0} in the case of affine coordinate planes.  It is of independent interest and we thus state this special case of the theorem  as a corollary.

\begin{cor}\label{corA}
Let $n > m  \ge 1 $ be integers and $d:= n-m$.  Let $\tau$ satisfy \eqref{tau} and given $\bm\beta=(\beta_1,\dots,\beta_m) \in \R^m$, let
\begin{equation}\label{Pi}
\Pi_{\bm\beta}:=\big\{(\alpha_1,\dots,\alpha_d,\beta_1,\dots,\beta_m): \bm\alpha=(\alpha_1,\dots,\alpha_d)\in\R^d\,\big\}\,\subset\,\R^n\, .
\end{equation}
Then
\begin{equation}\label{sv5}
\dim \mathcal{S}_{n}(\tau)\cap\Pi_{\bm\beta} ~\geq~ s:=\frac{n+1}{\tau+1}-m\,.
\end{equation}
Furthermore,
\begin{equation}\label{svb770}
  \cH^s(\mathcal{S}_{n}(\tau)\cap \Pi_{\bm\beta} )=\cH^s(\Pi_{\bm\beta})\,.
\end{equation}
\end{cor}

 %In particular, if $d=m=1$ the plane $\Pi_{\bm\beta}$ is simply a horizontal line in $\R^2$.

 \bigskip

\begin{proof}

 There is nothing to prove in the case $\tau=\tfrac1n$ since by Dirichlet's theorem, $\cS_n(\frac1n)=\R^n$. Hence, we will assume that $\tfrac1n<\tau< \tfrac1m$ and so by definition $ s < d$.  Hence, it follows that
 $$
 \cH^s(\Pi_{\bm\beta})  =  \infty \, .
 $$
 Next,  given $\tau$ and $\bm\beta \in \R^m$ let
$$
B(\bm\beta;\tau):=\{q\in\N:\max_{1\le i\le m}\|q\beta_i\|<q^{-\tau}\}\,,
$$
where $\|\cdot\|$ stands for distance from the nearest integer. Also, given a set $X \subseteq \R^{n}$ let $\pi_d (X) $ denote the orthogonal projection of $X$  onto the first $d$ coordinates of $\R^{n}$.   Then,  it is easily verified that  ${\bm\alpha} \in  \pi_d \big( \mathcal{S}_{n}(\tau)  \cap\Pi_{\bm\beta} \big) $  if and only if
\begin{equation}\label{bpq}
{\bm\alpha} \in B_{\textbf{p},q}:= \left\lbrace \vv x =(x_{1},\dots,x_{d})\in\mathbb{R}^d: \max_{1\le i\le d} \left|x_{i}-\frac{p_{i}}{q}\right|<q^{-\tau-1}\right\rbrace  \,
\end{equation}
for infinitely many integers $ q \in B(\bm\beta;\tau)$ and $\vv p=(p_1,\dots,p_n)\in\Z^n$; that is
\begin{equation}\label{shitBREXIT}
\pi_d \big( \mathcal{S}_{n}(\tau)  \cap\Pi_{\bm\beta} \big) :=\limsup_{q\to\infty,\ q\in B(\bm\beta;\tau)} \!\! B_{\vv p,q}  \, .
\end{equation}
The projection map $\pi_d$ is bi-Lipschitz and thus  the measure part \eqref{svb770} of the corollary follows on showing that
\begin{equation}\label{svb7700}
\cH^s\Big(\pi_d \big( \mathcal{S}_{n}(\tau)  \cap\Pi_{\bm\beta} \big) \Big)=  \infty \,.
\end{equation}
The dimension part \eqref{sv5} of the corollary follows directly from  \eqref{svb770} and the definition of Hausdorff dimension. This completes the proof of Corollary \ref{corA} modulo  \eqref{svb7700}.

We now establish \eqref{svb7700}.   Define $\eta$ to be the real number such that
\begin{equation}\label{vb01}
d\eta+m\tau=1.
\end{equation}
It is easily seen that $0<\eta<\tau$.  By Minkowski's theorem for systems of linear forms \cite{Schmidt80}, for any $\bm\alpha\in\R^d$ and for any integer $Q\ge 1$ there exists $(p_1,\dots,p_n,q)\in\Z^{n+1}\setminus\{\vv0\}$ such that
\begin{equation}\label{v2}
  \left\{\begin{array}{l}
|q\alpha_i-p_i|< Q^{-\eta}\quad\quad\ (1\le i\le d),\\[2ex]
|q\beta_{j}-p_{d+j}|<Q^{-\tau}\quad(1\le j \le m),\\[2ex]
|q|\le Q\,.
         \end{array}
  \right.
\end{equation}
Since both $\eta$ and $\tau$ are positive, we have that $Q^{-\eta}\le1$, $Q^{-\tau}\le1$ and we necessarily have that $q\neq0$. Furthermore, without loss of generality, we can assume that $q>0$.
Assume that $\alpha_1$ is irrational. Then, since $\eta>0$,
there must be infinitely many different $q\in\N$ amongst the solutions $(p_1,\dots,p_n,q)$ to \eqref{v2} taken over all $Q>1$. Indeed, if the same $q$ repeatedly occurred in the first inequality, the left hand side would  be a fixed positive constant for this $q$.   However, the right hand side of the first inequality tend to zero as $Q$ increases and we obtain a contradiction for $Q$ large.  Therefore, the system of inequalities
\begin{equation}\label{v2+}
  \left\{\begin{array}{l}
|q\alpha_i-p_i|< q^{-\eta}\quad\quad\ (1\le i\le d),\\[2ex]
|q\beta_{j}-p_{d+j}|<q^{-\tau}\quad(1\le j \le m)
         \end{array}
  \right.
\end{equation}
is satisfied  for infinitely many different $q\in\N$ and $(p_1,\dots,p_n)\in\Z^n$. This means that
$$
 \limsup_{q\to\infty,\ q\in B(\bm\beta;\tau)} \!\! B^s_{\vv p,q} \ \ \supseteq   \ \  \R^d\setminus\Q^d   \, ,  $$
where
\begin{equation}\label{s}
s:=\frac{(\eta+1)d}{\tau+1}  \ ~\stackrel{\eqref{vb01}}{=}  \  ~\frac{n+1}{\tau+1}-m
\end{equation}
and  the ball  $B^s_{\vv p,q}$   associated with $B_{\vv p,q}$ is defined via \eqref{e:006} with $k=d$.    Thus, for any ball  $B$ in $\R^d$
 it  trivially follows that
$$
 \cH^d \big( \ B\cap \limsup_{q\to\infty,\ q\in B(\bm\beta;\tau)} \!\! B^s_{\vv p,q} \ \big)=\cH^d(B)  \, .
$$
Now $s < d$ and so, by the Mass Transference Principle, we conclude that for any ball $B$ in $\R^d$
$$
\cH^s \big( \ B\cap \limsup_{q\to\infty,\ q\in B(\bm\beta;\tau)} \!\! B_{\vv p,q} \ \big)=\cH^s(B)=\infty  \, .
$$
This together with \eqref{shitBREXIT} implies  \eqref{svb7700}  and thereby completes the proof of the corollary.
\end{proof}

\bigskip

\begin{rem}
When it comes to establishing  Theorem~\ref{t0} (or rather Theorem  \ref{t1} below)   for  arbitrary  submanifolds $\cM $  of $\R^n$, it will be apparent that  the role of the affine coordinate plane $\Pi_{\bm\beta}$ in the above proof will be played by tangent planes to $\cM$ and the set $B(\bm\beta;\tau)$ will correspond to  rational points lying close to $\cM$.
\end{rem}

\section{The general case}

The following is a version of Theorem~\ref{t0} for general approximating functions $\psi$.

\begin{thm}\label{t1}
Let $\mathcal{M}$ be any $C^2$ submanifold of $\R^n$ of codimension $m$ and let $\psi:\N\to  \R^+$ be any monotonic function such that
\begin{equation}\label{v4}
\limsup_{Q\to\infty}Q^{\tfrac1m}\psi(Q)=\infty\,.
\end{equation}
Suppose that for some $\tau$ with $\frac1n\le\tau\le\frac1m$ we have that
\begin{equation}\label{v4m}
\inf_{Q\in\N}Q^\tau\psi(Q)>0\,.
\end{equation}
Then
\begin{equation}\label{v5m}
\dim \mathcal{S}_{n}(\psi)\cap\mathcal{M} ~\geq~ s:=\frac{n+1}{\tau+1}-m \,.
\end{equation}
Furthermore,
\begin{equation}\label{vb770m}
  \cH^s(\mathcal{S}_{n}(\psi)\cap\mathcal{M})=\cH^s(\cM)   \,.
\end{equation}
\end{thm}

\bigskip

Observe that any approximating function $\psi$ given by $\psi(q)=q^{-\tau}$ with $\frac1n\le\tau<\frac1m$ satisfies  conditions \eqref{v4}  and \eqref{v4m}. Thus, Theorem~\ref{t0} is  a simple consequence of Theorem~\ref{t1}. The following is another consequence of Theorem~\ref{t1} expressed in terms of \emph{the upper order at infinity of $1/\psi$} defined by
$$
\hat\tau_\psi:=\limsup_{Q\to\infty}\frac{-\log\psi(Q)}{\log Q}\,.
$$

\bigskip

\begin{cor}\label{cor1}
Let $\mathcal{M}$ be any $C^2$ submanifold of $\R^n$ of codimension $m$ and let $\psi:\N\to\R^+$ be any monotonic function satisfying \eqref{v4} and
\begin{equation}\label{v4o}
\inf_{Q\in\N}Q^{\tfrac1m}\psi(Q)>0\,.
\end{equation}
Suppose that $\hat\tau_\psi\ge\tfrac1n$. Then
\begin{equation}\label{v4*}
\dim\cS_n(\psi)\cap\cM\ge\frac{n+1}{\hat\tau_\psi+1}-m\,.
\end{equation}
\end{cor}

\bigskip

\begin{proof}
Note that, by \eqref{v4o}, $\hat\tau_\psi\le \tfrac1m$. First assume that $\hat\tau_\psi=\tfrac1m$. Let $\tau=\tau_\psi$. Then \eqref{v4m} becomes identical to \eqref{v4o} and Theorem~\ref{t1} is applicable. In this case \eqref{v4*} coincides with \eqref{v5m}.

Now assume that $\frac1n\le\hat\tau_\psi<\frac1m$ and fix any  $\tau$ such that $\hat\tau_\psi<\tau<\frac1m$. It readily follows from the definition of $\hat\tau_\psi$ that for all sufficiently large $q$ we have that
$$
\frac{-\log\psi(q)}{\log q}<\tau\,,
$$
or equivalently $\psi(q)>q^{-\tau}$ for all sufficiently large $q$. This implies the validity of \eqref{v4m} and thus Theorem~\ref{t1} is applicable. In turn,  since $\tau$ can be taken to be arbitrarily close to $\hat\tau_\psi$ the desired lower bound \eqref{v4*}  then follows from \eqref{v5m}.
\end{proof}

\bigskip

\begin{rem}
In the case $\hat\tau_\psi$ is strictly bigger than the lower order at infinity of $1/\psi$, that is
$$
\tau^*_\psi:=\liminf_{Q\to\infty}\frac{-\log\psi(Q)}{\log Q},
$$
\eqref{v4*} is not always sharp. For example, this is the case for non-degenerate submanifolds of $\R^n$ \cite{B12}, where the lower bound is shown to be
\begin{equation}\label{vbn}
\dim\cS_n(\psi)\cap\cM\ge \frac{n+1}{\tau^*_\psi+1}-m.
\end{equation}
\end{rem}

\bigskip

\begin{rem}  Observe that if $\tau < \frac1m$, then condition \eqref{v4m}  implies \eqref{v4}. In the case $ \tau=1/m$  condition \eqref{v4m}  implies  that
\begin{equation}\label{sv1}
\limsup_{Q\to\infty}Q^{\tfrac1m}\psi(Q) > 0 \,.
\end{equation}
However, this alone is not sufficient. To see this, consider the plane $\Pi_{\bm\beta}$ given by \eqref{Pi} with $\bm\beta$ being any badly approximable point in $\R^m$. This means that there exists a constant $c_0>0$ such that
$$
c_0:=\inf_{q\in \N}\,\,q^{\tfrac1m}\max_{1\le j\le m}\|q\beta_j\|>0\,.
$$
Let
\begin{equation}\label{v6}
\psi(q)=c_0\,q^{-\tfrac1m}\,.
\end{equation}
Clearly condition \eqref{sv1} is satisfied. Indeed, the corresponding limit exists and is strictly positive and finite  and thus $\tau_\psi=\frac1m$. However, by our choice of $\bm\beta$ and $\psi$, we have that
$$\Pi_{\bm\beta}\cap\cS_n(\psi)=\emptyset \, ,$$ and hence the conclusions of either the theorem or corollary are not valid.
\end{rem}

\bigskip

\begin{rem}
Despite the above remark, it should be noted that imposing additional `curvature' or Diophantine conditions on $\cM$ may help extending the range of $\psi$ for which \eqref{v4*} holds. For instance, it follows from \cite[Theorem~7.2]{B12} that for non-degenerate analytic curves \eqref{vbn} holds if $\frac1n\le\tau^*_\psi<\tfrac{3}{2n-1}$.
\end{rem}

\section{A Dirichlet type theorem for rational approximations to manifolds}

Let $\mathcal{M}$ be any $C^2$ submanifold of $\R^n$ of codimension $m$ and let $d:=n-m$ denote the dimension of $\mathcal{M}$.
Without loss of generality,  we will assume that $\cM$ is given by a Monge parametrisation, that is
\begin{equation}  \label{monge}
\cM:= \big\lbrace (\boldsymbol{\alpha}, \vv f(\boldsymbol{\alpha})) \in \mathbb{R}^n: \bm\alpha=(\alpha_{1},\dots, \alpha_{d}) \in \mathcal{U}\,\big\rbrace\,,
\end{equation}
where $\mathcal{U}$ is an open subset of $\mathbb{R}^d$ and where $\textbf{f}= (f_{1},\dots, f_{m})$ is defined and twice continuously differentiable on $\mathcal{U}$.
We furthermore assume without loss of generality that
\begin{equation}\label{D}
D:= \max_{\substack{ 1\leq j\le m\\ 1\leq i\leq d}}\ \underset{\boldsymbol{\alpha}\in \mathcal{U}}\sup \left| \frac{\partial f_{j}}{\partial \alpha_{i}}(\boldsymbol{\alpha}) \right| < \infty
\end{equation}
and
\begin{equation}\label{C}
C:= \max_{\substack{ 1\leq j\le m\\ 1\leq i,k\leq d}}\ \underset{\boldsymbol{\alpha}\in \mathcal{U}}\sup \left| \frac{\partial^2 f_{j}}{\partial \alpha_{i}\partial \alpha_{k}}(\boldsymbol{\alpha})\right| < \infty\,.
\end{equation}

\noindent The following result can be viewed as an analogue of Dirichlet's theorem for manifolds $\cM \subset \R^n$.  In short, the points of interest are restricted to  $\cM$ and given an approximating function $\psi$,  the rational points $\vv p/q  \in \Q^n$ lie within some  $\psi$-neighborhood of $\cM$.

\bigskip

\begin{thm}\label{Dir}
Let  $\mathcal{M}$ be as above and let $\psi:\N\to(0,1]$ be any monotonic function satisfying
\eqref{v4}
and
\begin{equation}\label{v4++}
\inf_{Q\in\N}Q\psi(Q)>0\,.
\end{equation}
Then for any  $\boldsymbol{\alpha}=(\alpha_{1},\dots,\alpha_{d})\in \mathcal{U}$ there is an infinite subset $\cQ\subset\N$ such that for any $Q\in\cQ$ there exists $(p_{1},\dots, p_{n},q) \in \mathbb{Z}^{n+1}$  with $1\leq q \leq Q$ and $(p_{1}/q,\dots,p_{d}/q) \in \mathcal{U}$ such that
\begin{equation}\label{v4.4}
\left| \alpha_{i} -\frac{p_{i}}{q}\right| < \frac{2^{m/d}}{q(Q\psi(Q)^m)^{1/d}}\qquad\text{for $1\leq i \leq d$}
\end{equation}
and
\begin{equation}\label{v4.5}
\left| f_{j}\left(\frac{p_{1}}{q},\dots,\frac{p_{d}}{q}\right) -\frac{p_{d+j}}{q}\right| < \frac{\psi(q)}{q}\qquad\text{for $1\leq j \leq m$.}
\end{equation}
\end{thm}

\bigskip

\begin{proof}
Define the following functions of $\boldsymbol{\alpha} \in \mathcal{U}$
\begin{equation}
g_{j}:= f_{j}- \sum_{i=1}^{d} \alpha_{i}\frac{\partial f_{j}}{\partial \alpha_{i}} \qquad (1\leq j \leq m)
\end{equation}
and consider the system of inequalities
\begin{empheq}[left=\empheqlbrace]{align}
    &  \Big| qg_{j}(\boldsymbol{\alpha}) +\sum_{i=1}^{d} p_{i}\frac{\partial f_{j}}{\partial\alpha_{i}}(\boldsymbol{\alpha}) -p_{d+j}  \Big| <\tfrac12\psi(Q)\qquad(1\leq j \leq m)  \label{Mink1}\\[1ex]
    &  \left| q\alpha_{i} - p_{i}\right| < (2^{-m}Q\psi(Q)^m)^{-1/d}\qquad\qquad\qquad\ \ (1\leq i \leq d)\label{Mink2}\\[2ex]
    &  |q| \leq Q\,. \label{Mink3}
  \end{empheq}
In the left hand side of this system of inequalities we have $(m+d+1)$ linear forms that correspond to the $(m+d+1)$ rows of the matrix
\begin{equation*}
G = G(\boldsymbol{\alpha}) :=\begin{pmatrix}
  g_{1} & \frac{\partial f_{1}}{\partial \alpha_{1}} & \cdots & \frac{\partial f_{1}}{\partial \alpha_{d}} & -1 & 0 & \cdots & 0 \\

  \vdots & \vdots & \ddots & \vdots & \vdots & & \ddots & \vdots \\

  g_{m}  & \frac{\partial f_{m}}{\partial \alpha_{1}}  & \cdots & \frac{\partial f_{m}}{\partial \alpha_{d}} & 0 & 0 & \cdots & -1  \\

  \alpha_{1} & -1 & \cdots & 0 & 0 & 0 & \cdots & 0   \\

  \vdots & \vdots & \ddots & \vdots & \vdots & \vdots & & \vdots \\

  \alpha_{d} & 0 & \cdots & -1 & 0 & 0 & \cdots & 0 \\

  1 & 0 & \cdots & 0 & 0 & 0 & \cdots & 0
 \end{pmatrix}.
\end{equation*}
The determinant of $G$ is easily seen to be $\pm1$, while the product of the quantities on right hand side of the system of inequalities \eqref{Mink1} -- \eqref{Mink3} is $1$. Therefore, by Minkowski's theorem for systems of linear forms,
there exists a non-zero solution $(p_{1},\dots, p_{n},q) \in \mathbb{Z}^{n+1}$ to the  system of inequalities \eqref{Mink1}--\eqref{Mink3}. Without loss of generality,  we will assume that $q\ge0$.  In other words,   inequality \eqref{Mink3} reads
\begin{equation}
\label{Mink3+}
0 \leq q \le Q \, .
\end{equation}
From now on we fix $\bm\alpha\in\cU$ and let
$$B:=\inf_{Q\in\N}Q\psi(Q)   \stackrel{\eqref{v4++}}{>}    0 \, . $$  Furthermore, since $\cU$ is open there exists a ball $B(\bm\alpha,r)$ centered at $\bm\alpha$ of radius $r$ which is contained in $\cU$. Define
\begin{equation}\label{cQ}
\cQ:=\Big\{Q\in\N:(2^{-m}Q\psi(Q)^m)^{-1/d}< \min\big\{1,r,\left(\frac{B}{Cd^2}\right)^{\frac12}\big\}\Big\}\,,
\end{equation}
where $C$ is given by \eqref{C}.
Then, by \eqref{v4}, $\cQ$ is an infinite subset of $\N$.

The upshot of the above is that the proof of  Theorem~\ref{Dir} is reduced to  showing  that for any $Q\in\cQ$ the solution $(p_{1},\dots, p_{n},q)$ to the system of inequalities  \eqref{Mink1}, \eqref{Mink2} and \eqref{Mink3+} arising from Minkowski's theorem is in fact a solution to the system of inequalities associated with  Theorem~\ref{Dir}.  With this in mind, we first show that we can  take $ q \ge 1 $ in  \eqref{Mink3+}.   Indeed, suppose $q$ vanishes. Then, by \eqref{cQ}, we have that $(2^{-m}Q\psi(Q)^m)^{-d}< 1$. Hence, inequalities (\ref{Mink2}) imply that $\left| p_{i} \right|< 1$ for each $1\leq i \leq d$, and as $p_{i} \in \mathbb{Z}$ we must therefore have $p_{i}=0$ for $1\le i\le d$. This together with   (\ref{Mink1}) implies that $\left| p_{d+j}\right| <\psi(Q)<1$ and so, by the same reasoning, $\left| p_{d+j}\right| =0$ for $1\leq j \leq m$. Hence $(p_{1},\dots, p_{n},q)= \textbf{0}$, contradicting the fact that the solution guaranteed by Minkowski's theorem is non-zero. Thus, $q\ge1$ as claimed. In turn, on dividing (\ref{Mink2}) by $q$, we obtain \eqref{v4.4} and  by the definition of $\cQ$  the associated  rational point $\vv p/q:= (p_{1}/q,\dots,p_{d}/q)\in B(\bm\alpha,r)\subset\cU$.

It remains to show that \eqref{v4.5} is satisfied.  By Taylor's formula, for each $1 \le j \le m$, we have that
\begin{equation}
f_{j}\left(\frac{\textbf{p}}{q} \right) = f_{j}(\boldsymbol{\alpha}) + \sum_{n=1}^{d} \alpha_{i}\frac{\partial f_{j}}{\partial \alpha_{i}}(\boldsymbol{\alpha})\left(\frac{p_{i}}{q}-\alpha_{i}\right) + R_{j}(\boldsymbol{\alpha},\tilde{\boldsymbol{\alpha}})
\end{equation}
for some $\tilde{\boldsymbol{\alpha}}$ (depending on $\bm\alpha$ and $\vv p/q$) which lies on the line joining $\boldsymbol{\alpha}$ and $\textbf{p}/q$, where
\begin{equation}\label{v4.11}
R_{j}(\boldsymbol{\alpha},\tilde{\bm\alpha}):= \frac{1}{2}\sum_{i=1}^{d} \sum_{k=1}^{d}\frac{\partial^2 f_{j}}{\partial\alpha_{i} \partial\alpha_{k}}(\tilde{\bm\alpha})\left(\frac{p_{i}}{q}-\alpha_{i}\right)\left(\frac{p_{k}}{q}-\alpha_{k}\right).
\end{equation}
In particular, since $\vv p/q\in B(\bm\alpha,r)$ and a ball is a convex set,  $\tilde{\boldsymbol{\alpha}}\in B(\bm\alpha,r)\subset\mathcal{U}$ and so $R_{j}(\boldsymbol{\alpha},\tilde{\boldsymbol{\alpha}})$ is well defined.
Using \eqref{v4.11}, we now rewrite the left hand side of (\ref{Mink1}) in the following way:
\begin{align}
\Big| qg_{j}(\boldsymbol{\alpha}) & + \sum_{i=1}^{d} p_{i}\frac{\partial f_{j}}{\partial\alpha_{i}}(\boldsymbol{\alpha}) -p_{d+j}  \Big|\nonumber\\[1ex]
       =& \ \left|q\left(f_{j}(\boldsymbol{\alpha})- \sum_{i=1}^{d} \alpha_{i}\frac{\partial f_{j}}{\partial \alpha_{i}}(\boldsymbol{\alpha})\right) +\sum_{i=1}^{d} p_{i}\frac{\partial f_{j}}{\partial\alpha_{i}}(\boldsymbol{\alpha}) -p_{d+j}  \right|\nonumber\\[1ex]
       = &   \  \Big|qf_{j}(\boldsymbol{\alpha})+ \sum_{i=1}^{d} (p_{i}-q\alpha_{i})\frac{\partial f_{j}}{\partial\alpha_{i}}(\boldsymbol{\alpha}) -p_{d+j}  \Big|\nonumber \\[2ex]
       =&    \  \left| q \left(f_{j}\left(\frac{\textbf{p}}{q} \right)- \sum_{n=1}^{d} \alpha_{i}\frac{\partial f_{j}}{\partial \alpha_{i}}(\boldsymbol{\alpha})\left(\frac{p_{i}}{q}-\alpha_{i}\right) - R_{j}(\boldsymbol{\alpha},\tilde{\boldsymbol{\alpha}})\right) \right. \nonumber\\[2ex]
        &\qquad\qquad\qquad\qquad +   \left. \sum_{i=1}^{d} (p_{i}-q\alpha_{i})\frac{\partial f_{j}}{\partial\alpha_{i}}(\boldsymbol{\alpha}) -p_{d+j}\right| \nonumber \\[1ex]
       =&   \  \left|q f_{j}\left(\frac{\textbf{p}}{q} \right)-p_{d+j} - qR_{j}(\boldsymbol{\alpha},\tilde{\boldsymbol{\alpha}})\right|.\label{vb37}
\end{align}

\noindent  Suppose for the moment that
\begin{equation}\label{Boundsv}
\left|qR_{j}(\boldsymbol{\alpha},\tilde{\boldsymbol{\alpha}})\right| < \frac{\psi(q)}{2}\,.
\end{equation}
Then together with \eqref{Mink1} and  \eqref{vb37}, we obtain that
\begin{align}
\left|qf_{j}\left(\frac{\textbf{p}}{q}\right) -p_{d+j} \right| & \ \leq  \ \left|q f_{j}\left(\frac{\textbf{p}}{q} \right)-p_{d+j} - qR_{j}(\boldsymbol{\alpha},\tilde{\boldsymbol{\alpha}})\right| + \left|qR_{j}(\boldsymbol{\alpha},\tilde{\boldsymbol{\alpha}})\right|\nonumber\\[2ex]
& \  < \
\frac{\psi(Q)}{2} +\frac{\psi(q)}{2},\label{Bound4}
\end{align}
and so, by the monotonicity of $\psi$, we conclude that
\begin{equation*}
\left|qf_{j}\left(\frac{\textbf{p}}{q}\right) -p_{d+j} \right| < \psi(q)\,.
\end{equation*}
This implies \eqref{v4.5} and thereby completes the proof of Theorem~\ref{Dir} modulo the truth of \eqref{Boundsv}.

In order to establish  \eqref{Boundsv},   we   use \eqref{C} and \eqref{v4.4} to  obtain that
\begin{align*}
 \nonumber\left|qR_{j}(\boldsymbol{\alpha},\tilde{\boldsymbol{\alpha}})\right| & \ =  \
 \left|\frac{q}{2}\sum_{i=1}^{d} \sum_{k=1}^{d}\frac{\partial^2 f_{j}}{\partial\alpha_{i} \partial\alpha_{k}}(\tilde{\boldsymbol{\alpha}})\left(\frac{p_{i}}{q}-\alpha_{i}\right)\left(\frac{p_{k}}{q}-\alpha_{k}\right)\right| \\[2ex]
  & \ <  \ \frac{Cqd^2}{2}\left(\frac{2^{\frac{m}{d}}}{q(Q\psi(Q)^m)^{\frac{1}{d}}}\right)^2. \label{Bound2}
\end{align*}
Thus \eqref{Boundsv} follows if
\begin{equation}\label{Bound3+}
\frac{Cqd^2}{2}\left(\frac{2^{\frac{m}{d}}}{q(Q\psi(Q)^m)^{\frac{1}{d}}}\right)^2 < \frac{\psi(q)}{2}\,.
\end{equation}
The latter is true if and only if
\begin{equation}\label{Bound3}
2^{\frac{2m}{d}}Cd^2 < q\psi(q)(Q\psi(Q)^m)^{\frac{2}{d}}
\end{equation}
and, in view of the definition of $B$, is implied by the inequality
$$
2^{\frac{2m}{d}}Cd^2 < B(Q\psi(Q)^m)^{\frac{2}{d}}.
$$
The latter however is true for any $Q\in\cQ$ -- see \eqref{cQ}. Therefore \eqref{Bound3} and consequently \eqref{Bound3+} are satisfied.  This establishes \eqref{Boundsv} as desired.
\end{proof}

\bigskip

The following  `infinitely often'  consequence  of Dirichlet's theorem for manifolds  will be the key to establishing Theorem~\ref{t1}.

\begin{cor}\label{Cor2}
Let  $\mathcal{M}$ and  $\psi$ be as in Theorem~\ref{Dir}.   Furthermore, let $\tau\le\frac1m$ and suppose there exists a constant $\kappa>0$  such that
\begin{equation}\label{v4*+}
\kappa\le Q^\tau\psi(Q)
\end{equation}
for infinitely many $Q\in\N$.
Then, for any $\boldsymbol{\alpha}\in \cU\setminus\mathbb{Q}^d$ there exist infinitely many different vectors $(p_{1},\dots, p_{n},q) \in \mathbb{Z}^{n}\times\N$ with $(p_{1}/q,\dots,p_{n}/q) \in \mathcal{U}$ satisfying the system of inequalities \eqref{v4.5} and
\begin{equation}\label{v4.21}
\left| \alpha_{i} -\frac{p_{i}}{q}\right| < \left(\frac2\kappa\right)^{\frac md}q^{-1-\frac{1-\tau m}{d}}\qquad\text{for \  $1\leq i \leq d$. }
\end{equation}
\end{cor}

\begin{proof}
Fix any $\bm\alpha\in\cU$ and let $\cQ$ be given by \eqref{cQ}. Then, as was shown in the proof of Theorem~\ref{Dir}, for any $Q\in\cQ$ there exists $(p_1,\dots,p_n,q)\in\Z^n\times\N$ with $q\le Q$,
$(p_{1}/q,\dots,p_{n}/q) \in \mathcal{U}$ and such that \eqref{v4.4} and \eqref{v4.5} hold. Let $S$ be the set of $Q\in\N$ satisfying \eqref{v4*+}. By the conditions of the corollary, $S$ is infinite. Clearly $Q\psi(Q)^m\ge Q^{1-m\tau}\kappa$ for $Q\in S$. Hence, $Q\psi(Q)^m$ is unbounded on $S$ if $\tau<1/m$. By \eqref{v4}, $Q\psi(Q)^m$ is unbounded on $S$ in the case $\tau=1/m$ as well. Choose, as we may, any strictly increasing sequence $(Q_i)_{i\in\N}$ of numbers from $S$ such that $\lim_{i\to\infty}Q_i\psi(Q_i)^m$. By the definition \eqref{cQ} of $\cQ$, there is a sufficiently large $i_0$ such that $Q_i\in\cQ$ for all $i\ge i_0$. Define $\cQ^*=\{Q_i:i\ge i_0\}$. Then, by construction,
\eqref{v4*+} holds for any $Q\in\cQ^*$ and
\begin{equation}\label{donald}
\sup_{Q\in\cQ^*}Q\psi(Q)^m=\infty  \, .
\end{equation}
By \eqref{v4.4} and \eqref{v4*+}, we have that for any $Q\in\cQ^*$  the associated  solution $(p_1,\dots,p_n,q)\in \Z^n\times\N $ satisfies
\begin{align*}
\left| \alpha_{i} -\frac{p_{i}}{q}\right| & \ < \ \frac{2^{m/d}}{q(Q\psi(Q)^m)^{1/d}}
\le \frac{2^{m/d}}{q(Q\kappa^mQ^{-\tau m})^{1/d}}\\[3ex]
& \  =   \ \left(\frac2\kappa\right)^{\frac{m}{d}}\frac{1}{q\,Q^{\frac{1-\tau m}{d}}}
~\le~ \left(\frac2\kappa\right)^{\frac{m}{d}}\frac{1}{q\cdot q^{\frac{1-\tau m}{d}}} \qquad\text{for \  $1\leq i \leq d$. }
\end{align*}
This  is exactly \eqref{v4.21} as desired. To complete the proof it suffices to show that there are infinitely many different $q$'s among the solutions $(p_1,\dots,p_n,q)$ as $Q$ runs through $\cQ^*$.   With this in mind, suppose on the contrary that there are  only finitely many such $q$'s and  let $A$ be the corresponding set.  As $\boldsymbol{\alpha}\in \cU\setminus\mathbb{Q}^d$, there exists some $1\leq i\leq d$ for which  $\alpha_{i} \not\in \mathbb{Q}$. Hence, there exists some $\delta_0>0$ such that
\begin{equation*}
\delta_0\le \min_{\substack{q\in A\\ p_i\in\Z}}\left|q\alpha_{i}-p_{i}\right|   \, .
\end{equation*}
Together with  \eqref{v4.4}, it follows that
\begin{equation*}
\delta_0\le\left|q\alpha_{i}-p_{i}\right|\leq \frac{2^{\frac{m}{d}}}{(Q\psi(Q)^m)^{\frac{1}{d}}}\,.
\end{equation*}
This contradicts   \eqref{donald} and thereby completes the proof of the corollary.
\end{proof}

\section{Proof of the main theorem}

 Let $\psi$ be as in Theorem~\ref{t1} and let $\kappa$ denote the infimum defined by \eqref{v4m}. We can  assume without loss of generality that $\psi: \N \to (0,1]$.  Indeed, if this was not the case we define  the auxiliary function $\tilde{\psi}: q  \to \tilde{\psi}(q):= \min\{1, \psi(q)\}$ and since   $\mathcal{S}_{n}(\tilde{\psi})  \subseteq \mathcal{S}_{n}(\psi)$,  it  suffices to prove the theorem with $\psi$ replaced by $ \tilde{\psi}$.

\bigskip

 The following lemma from \cite{note} will be required during the course  of the proof of the theorem.

\begin{lem}\label{lem1}
Let $\lbrace B_{i}\rbrace$ be a sequence of balls in $\mathbb{R}^k$ with $\vert B_{i}\vert\rightarrow 0$ as $i\rightarrow \infty$. Let $\lbrace U_{i}\rbrace$ be a sequence of Lebesgue measurable sets such that $U_{i}\subset B_{i}$ for all $i$. Assume that for some $c>0$, $\vert U_{i} \vert \geq c\vert B_{i}\vert$ for all $i$. Then the sets $\mathcal{C}:= \limsup_{i \to \infty} U_{i}$ and $\mathcal{B}:= \limsup_{i\to \infty} B_{i}$ have the same Lebesgue measure.
\end{lem}

\bigskip

Without loss of generality,  we  assume that $\cM$ is given \eqref{monge}.
Now, given $ \vv f $  and $\psi$ let
 \begin{equation}\label{pq}
B(\vv f, \psi) :=  \left\{ (\vv p,q) \in \mathbb{Z}^{n}\times\N : (p_{1}/q,\dots,p_{d}/q) \in \mathcal{U} \text{\  and \eqref{v4.5} holds.} \right\}
\end{equation}
  In view of   Corollary~\ref{Cor2},  we have that
for almost every $\bm\alpha\in\cU$ (in fact all irrational $\bm\alpha\in\cU$)
there are infinitely many different vectors  $(\vv p,q) \in  B(\vv f, \psi)  $ satisfying  \eqref{v4.21}.
Then, for any fixed $\delta>0$, it follows via Lemma~\ref{lem1} that for almost every $\bm\alpha\in\cU$
there are infinitely many different vectors $(\vv p,q) \in  B(\vv f, \psi)  $ satisfying
\begin{equation}\label{v4.21x+}
\left| \alpha_{i} -\frac{p_{i}}{q}\right| < \delta q^{-1-\eta}\qquad\text{for $1\leq i \leq d$}\,,
\end{equation}
where $\eta$ is as in  \eqref{vb01}.

Next, for $(\vv p,q) \in  B(\vv f, \psi)  $ consider the ball
\begin{equation}\label{bpq+}
B_{\textbf{p},q}= \left\lbrace\vv x=(x_{1},\dots,x_{d})\in\mathbb{R}^d: \max_{1\le i\le d}\left|x_{i}-\frac{p_{i}}{q}\right|<\delta^{d/s}q^{-\tau-1}\right\rbrace,
\end{equation}
where $s$ is given by \eqref{v5m}.
The previous statement concerning \eqref{v4.21x+} implies that for any ball  $B$ in $\cU\subset\R^d$ we have that
$$
 \cH^d \big( \ B\cap \limsup_{(\vv p,q)\in B(\vv f,\psi)} \!\! B^s_{\vv p,q} \ \big)=\cH^d(B)  \, ,
$$
where the ball  $B^s_{\vv p,q}$   associated with $B_{\vv p,q}$ is defined via \eqref{e:006} with $k=d$.
Let
$$
W:=\limsup_{(\vv p,q)\in B(\vv f,\psi)} \!\! B_{\vv p,q}\,.
$$
Then, by the Mass Transference Principle, for any ball $B$ in $\cU$ it trivially follows that
\begin{equation}\label{W}
\cH^s \big( \ B\cap W \ \big)=\cH^s(B)\,.
\end{equation}

\noindent Now let $\bm\alpha\in B_{\vv p,q}$ for some $(\vv p,q)\in B(\vv f,\psi)$. Then on using the triangle inequality and the Mean Value Theorem, for each $1\le j\le m$ and some $\textbf{c}\in \mathcal{U}$,  it follows  that
\begin{align}
\nonumber \left|f_{j}\left(\boldsymbol{\alpha}\right) -\frac{p_{d+j}}{q} \right| & \leq \left|f_{j}\left(\boldsymbol{\alpha}\right) - f_{j}\left(\frac{\textbf{p}}{q}\right)\right| + \left|f_{j}\left(\frac{\textbf{p}}{q}\right) -\frac{p_{d+j}}{q} \right|\\[2ex]
\nonumber &=\left|\nabla f_{j}(\textbf{c})\cdot\left(\boldsymbol{\alpha}-\frac{\textbf{p}}{q}\right)\right|+
\left|f_{j}\left(\frac{\textbf{p}}{q}\right) -\frac{p_{d+j}}{q} \right|\\[2ex]
\nonumber &~\hspace*{-4ex} \stackrel{\eqref{D}\&\eqref{v4.5}}{\le} dD\max_{1\le i\le d}\left|\alpha_i-\frac{p_i}q\right|+\frac{\psi(q)}{q}\\[2ex]
\nonumber&\stackrel{\eqref{bpq+}}{\le} \delta^{d/s} dD q^{-1-\tau}+\frac{\psi(q)}{q}\\[2ex]
\nonumber&\le \delta^{d/s} dD\kappa^{-1} \,\frac{\psi(q)}{q}+\frac{\psi(q)}{q}\le\frac{2\psi(q)}{q}
\end{align}
provided $\delta > 0$ is sufficiently small.    In view of \eqref{v4m},  we can also  ensure that $\delta$ is  small enough  so that
$$
B_{\textbf{p},q}\subset \left\lbrace\vv x=(x_{1},\dots,x_{d})\in\mathbb{R}^d: \max_{1\le i\le d}\left|x_{i}-\frac{p_{i}}{q}\right|<\frac{\psi(q)}{q}\right\rbrace\,.
$$
Hence, in view of the  definition of $W$, it follow that
$$
W\subset\pi_d(\cS_n(2\psi)\cap\cM)\,.
$$
The projection map $\pi_d$ is bi-Lipschitz and thus it follows from  \eqref{W} that
\begin{equation}\label{vb770m++}
  \cH^s(\mathcal{S}_{n}(2\psi)\cap\mathcal{M})=\cH^s(\cM)   \,.
\end{equation}
Using \eqref{vb770m++} with $\frac12\psi$ (instead of $\psi$)  implies the measure part \eqref{vb770m} of the theorem.
The dimension part \eqref{v5m} of the theorem follows directly from  \eqref{vb770m} and the definition of Hausdorff dimension. This completes the proof of Theorem \ref{t1}.

\section{Concluding comments}

The new approach developed in this paper is simple and easy to apply.  It is based on first establishing an appropriate  Dirichlet-type result (assuming it does not already exist) and then  applying the Mass Transference Principle.  Within the context of Diophantine approximation on manifolds,  our approach enables us to establish lower bound dimension results for  any $C^2$ submanifold of $\R^n$.  This is  well beyond the class of non-degenerate manifolds for which general results are perceived to hold.   However, there is  a cost. The new  approach  does not allow us to  obtain the stronger  divergent Jarn\'ik-type results for  $ \cH^s(\cS_n(\psi)\cap\cM) $.  For this we require the significantly more sophisticated `ubiquity' approach developed in \cite{BDV}.

 \medskip

 We end the paper with several natural problems that explore the scope  of the new approach.

\medskip

 \noindent {\em Problem 1. }   Within the context of Corollary  \ref{cor1}, the new approach as implemented  requires the existence of the order $\tau_\psi$ at infinity of $1/\psi$.   It is not immediately clear whether or not it is possible to get away with  the weaker notion of the lower order at infinity of $1/\psi$ as in the modern version of the classical Jarn\'ik-Besicovitch theorem -- see \eqref{JBdani}.

   \noindent {\em Problem 2. }  Within the context of Theorem \ref{t0}, it would be interesting to know if the range of $\tau$  given by \eqref{tau} can be  extended to a larger range by imposing additional `mild' constraints on the $C^2$ submanifolds of $\R^n$.

     \noindent  {\em Problem 3. }  We suspect  that the  new approach can  successfully  be  applied to problems within the  setup of weighted approximations on manifolds.   More precisely,  it should be possible to give a simpler proof of the lower bound statement appearing in  Theorem~4  of \cite{BV07} and, at the same time, broaden the class of planar curves under consideration.

     \noindent {\em Problem 4. } It would be highly desirable to be able to  apply the new approach to inhomogeneous  problems. The goal would be to obtain inhomogeneous Jarn\'ik-Besicovitch type results for manifolds.  This seems to be a much harder task than the previous `weighted' problem.  In the first instance it might be useful to consider the case of planar curves and see if the new  approach can be utilised to give a simpler proof of the lower bound statement appearing in \cite[Corollary~1]{BVV11}.

   \noindent  {\em Problem 5. } In this paper we have completely restricted our attention to simultaneous approximation.   In \cite{DD00},  Dickinson $\&$ Dodson consider the problem of establishing a Jarn\'ik-Besicovitch type result for  dual approximation on manifolds.  In short, they prove a lower bound statement \cite[Theorem 2]{DD00} that is valid  for any extremal manifold.  We suspect that the  new approach can  be applied to problems within the dual setup.  Indeed, it may be possible to show that  Theorem~2  of \cite{DD00} is in fact valid well beyond the class of  extremal manifolds, possibly to all $C^2$ submanifolds of $\R^n$.

\vspace{4ex}

\noindent\textbf{Acknowledgements.}   The authors are grateful to the anonymous reviewer of the paper for their valuable suggestions. SV would like to thank Sharifa Ansari  for relieving him of excruciating  pain.  Her attitude and generosity  in a world full of immoral and selfish money grabbers has been most refreshing and humbling. Her actions have been a timely reminder  that there is always hope for humanity  --  even after the  political disasters of 2016!

\

\

{\footnotesize

\begin{minipage}{0.9\textwidth}
\footnotesize V. Beresnevich\\
University of York, Heslington, York, YO10 5DD, England\\
{\it E-mail address}\,:~~ \verb|victor.beresnevich@york.ac.uk|\\
\end{minipage}

\begin{minipage}{0.9\textwidth}
\footnotesize L. Lee\\
University of York, Heslington, York, YO10 5DD, England\\
{\it E-mail address}\,:~~ \verb|ldl503@york.ac.uk|\\
\end{minipage}

\begin{minipage}{0.9\textwidth}
\footnotesize R. C. Vaughan\\
Department of Mathematics, McAllister Building, Pennsylvania State University, University
Park, PA 16802-6401, U.S.A.\\
{\it E-mail address}\,:~~ \verb|rcv4@psu.edu|\\
\end{minipage}

\begin{minipage}{0.9\textwidth}
\footnotesize S. Velani\\
University of York, Heslington, York, YO10 5DD, England\\
{\it E-mail address}\,:~~ \verb|sanju.velani@york.ac.uk|\\
\end{minipage}

}

\end{document}